\documentclass[11pt]{amsart}
\usepackage{graphicx} 
\usepackage{amsmath,amsthm}
\usepackage{amsfonts}
\usepackage[margin=1.2in]{geometry}

\newtheorem{thm}{Theorem}
\newtheorem{theorem}[thm]{Theorem}
\newtheorem{defn}[thm]{Definition}
\newtheorem{definition}[thm]{Definition}

\newtheorem{lemma}[thm]{Lemma}

\newtheorem{example}[thm]{Example}

\newtheorem{quest}[thm]{Question}

\def\S{\mathcal{S}}
\def\G{\mathcal{G}}
\def\tr{\operatorname{tr}}
\def\Mn{\mathcal{M}_n(\mathbb{C})}
\def\M{\mathcal{M}}

\def\ep{\varepsilon}
\def\H{\mathcal{H}}
\def\bfep{\underline{\varepsilon}}
\def\epwas{\underline{\varepsilon}-\widehat{\mathrm{as}}}
\def\epws{\underline{\varepsilon}-\widehat{\mathrm{s}}}

\title{Groups of matrices with approximately submultiplicative spectra}
\author[M.~Mastnak]{Mitja Mastnak}
\address{Department of Mathematics and Computing Science, Saint Mary's University, 923 Robie St, Halifax, Nova Scotia, Canada B3N 1Z9}
\email{mitja.mastnak@smu.ca}
\author[L.~McNamara]{Lindsey McNamara}
\address{Department of Mathematics and Computing Science, Saint Mary's University, 923 Robie St, Halifax, Nova Scotia, Canada B3N 1Z9}
\email{lindsey.mcnamara@smu.ca}
\author[Z.~Yu]{Zhipeng Yu}
\address{Department of Mathematics and Computing Science, Saint Mary's University, 923 Robie St, Halifax, Nova Scotia, Canada B3N 1Z9}
\email{zhipeng.yu@smu.ca}

\date{\today}


\begin{document}

\maketitle

\begin{abstract}
We say that a semigroup of matrices has a submultiplicative spectrum if the spectrum of the product of any two elements of the semigroup is contained in the product of the two spectra in question (as sets).  In this note we explore an approximate version of this condition.  
\end{abstract}

\section{Introduction}
A semigroup $\S$ of complex $n\times n$ matrices has a submultiplicative spectrum if for every pair $A,B\in\S$, every eigenvalue of the product $AB$ is equal to a product of an eigenvalue of $A$ and an eigenvalue of $B$.  This property was introduced by Lambrou, Longstaff, and Radjavi in 1992 \cite{LLR} and has since been extensively studied.  It has led to numerous nice structure results for matrix groups and semigroups.  Irreducible semigroups with this property are essentially finite nilpotent groups \cite[Thm. 3.3.4, Thm. 3.3.5]{RR} (see also \cite{LLR}).  In the original paper \cite{LLR}, the authors proved that such irreducible groups exist in all odd dimensions.  Kramar considered even dimensions in \cite{K1} and showed that for even $n$, there exist irreducible groups of $n\times n$ matrices with submultiplicative spectra if and only if $n$ is divisible by $8$.
In \cite{Om} the structure of irreducible $2$-groups was studied and in \cite{K2} this study was extended to include $p$-groups for general $p$.  In \cite{K3} the class ($\hat{\mathrm{s}}$) of groups $G$ with the property that all their irreducible sub-representations are submultiplicative was introduced.  A systematic study of this class of finite groups was initiated in \cite{GKOR}.

In this note, we start the study of an approximate version of submultiplicativity.  The idea of replacing exact conditions with approximate ones has a rich (recent) history.  The first ``approximate version" result in the context of simultaneous triangularization of matrix semigroups is \cite{BR}.  There, the authors ask ``how small can the spectra of nonzero commutators in a unitary group of matrices be?"  They show that for a non-commutative unitary matrix group $\G$, there always exist elements $A,B\in\G$ such that the spectral radius $\rho(AB-BA)$ of their ring commutator is at least $\sqrt{3}$.  Alternatively, this means that for unitary groups, the exact triangularizing condition (actually diagonalizing in this case) 
$$\rho(AB-BA)=0$$ 
and be replaced by an equivalent approximate condition 
$$\rho(AB-BA)<\sqrt{3}.$$  Since then a number of analogous results have been proven.  For example, in \cite{Ball}, the authors prove, among other things, that the exact triangularizing condition for semigroups of matrices 
$$\tr(ABC-BAC)=0$$ 
can be replaced by an approximate version 
$$|\tr(ABC-BAC)|<3$$ 
for unitary groups.  They also explore to what extent a similar result can apply to more general semigroups of matrices.  In \cite{KMOR} the authors show (using Chabauty topology) that for unitary groups $\G$ in 
$\mathcal{M}_n(\mathbb{C})$ every continuous multi-variate triangularizing condition 
$$
f(A_1,\ldots, A_k)=0
$$
can be replaced by an approximate version
$$
|f(A_1,\ldots, A_k)|\le \varepsilon_{f,n}
$$
for some $\varepsilon_{f,n}>0$ that depends on $f$ and $n$.  Naturally, in concrete cases, it can also be important to find the largest possible $\varepsilon_{f,n}$.  In particular, the question of whether 
$\varepsilon_{f,n}$ is bounded below by some polynomial in 
$\frac{1}{n}$ is of interest.

We say that a semigroup $\S$ of complex $n\times n$ matrices is $\ep$-submuliplicative if for all $A,B\in \S$ and for every eigenvalue $\gamma$ of $AB$, there exist eigenvalues $\alpha$ of $A$ and $\beta$ of $B$ such that
$$
|\gamma-\alpha\beta|\le \ep\rho(A)\rho(B).
$$
It is almost immediate that irreducible such semigroups do not contain any nonzero nilpotent elements.  However, such semigroups need not be groups.  For any value of $\ep>0$, there are irreducible semigroups of matrices of rank at most $1$, that satisfy this condition.  In this note we focus primarily on unitary groups. In this case, the condition becomes
$$
|\gamma-\alpha\beta|\le \ep.
$$
It also turns out that it is more convenient to consider a multiplicative analogue of this condition.  We say that a group $\G$ of unitary $n\times n$ matrices is $\ep$-argument-submultiplicative (or $\ep$-ASM) if for every pair $A,B\in\G$ and every eigenvalue $\gamma$ of $AB$, there exist eigenvalues $\alpha$ of $A$ and $\beta$ of $B$ such that 
$$
\frac{1}{2\pi}\left|\arg\left(\frac{\alpha\beta}{\gamma}\right)\right|\le\ep.
$$
(Here $\arg(z)\in (-\pi,\pi]$ denotes the principal argument of $z$.) We show, among other things, that for a unitary group, and $\varepsilon=\frac{1}{2n^2}$, the $\ep$-argument-submultiplicativity
implies that the group is finite modulo its centre.  We show, by example, that the order of $\varepsilon$ above is sharp: there exist
$\frac{2}{n^2}$-ASM groups of unitary matrices that are not finite modulo their centre.  We do not know at this point whether the quadratic order is still sharp if we additionally assume irreducibility.  More precisely: there are easy examples of irreducible $\ep$-ASM groups that are not essentially finite for $\ep>\frac{1}{2n}$, but we do not know if such examples exist when $\ep=\frac{c}{n^2}$ for some fixed constant $c$.

Our approximate results are quite different from those in the literature.  Most similar sounding results deal with studying conditions that imply reducibility or triangularizability of matrix collections (groups, semigroups, etc.).  In the case of submultiplicativity (and its approximate version) the situation is different. We mainly deal with irreducible (or at least completely reducible) groups and semigroups; the condition in question then implies something about their structure.  Another important distinction is that in most other situations in the literature, the approximate condition actually implies the exact condition for small $\ep$ (and not much is said about the structure when $\ep$ is not sufficiently small to ensure this).
In our paper,
this is not the case.  Even for an $\ep$ that is sufficiently small to
imply structure results (e.g., essential finiteness)
we have interesting examples of matrix groups that are approximately submultiplicative, but not submultiplicative.

\section{Preliminaries}
We start by reminding the reader of some definitions that we briefly mentioned in the introduction.  Throughout the section $\S\subseteq\Mn$ will denote a semigroup of complex $n\times n$ matrices and $\G\subseteq\Mn$ will denote a group of unitary matrices.  We say that $\S$ is $\ep$-submultiplicative if for all $A,B\in\S$ and every eigenvalue $\gamma$ of $AB$, there exist eigenvalues $\alpha$ of $A$ and $\beta$ of $B$ such that
$$
|\gamma-\alpha\beta|\le \ep\rho(A)\rho(B).
$$
Recall that a semigroup $\mathcal{S}$ is \emph{essentially finite} if $\mathcal{S}\subseteq \mathbb{C}\mathcal{S}_0$ where $\mathcal{S}_0$ is some finite semigroup. A semigroup is \emph{irreducible} if it has no nontrivial invariant subspaces, and \emph{reducible} otherwise. It has been shown \cite[Thm. 3.3.4.]{RR} that for an irreducible semigroup $\mathcal{S}$ with submultiplicative spectrum, $\mathcal{S}\setminus \{0\}$ is an essentially finite group. The following example will show that this does not hold for approximately submultiplicative semigroups.
\begin{example}
    Define the semigroup $\mathcal{S}_r$ of matrices for $0 < r < 1$ as 
    \[\mathcal{S}_r = \left\{\lambda\begin{pmatrix}
    1 & x^*\\
    y & yx^*
\end{pmatrix}: x,y \in \mathbb{C}^{n-
1}, ||x||, ||y|| < r, \lambda\in \mathbb{C}\right\}.\]
\end{example}
In \cite[Ex. 2.3.]{BR} it has been shown that $\mathcal{S}_r$ is irreducible. It is easy to see that $\S_r$ cannot be essentially finite.  Indeed, it is obvious that two distinct elements of $\S_r$ whose 
$(1,1)$-entries are both equal to $1$ cannot be multiples of each other and (again obviously) there are infinitely many pairwise-distinct such elements.  It is fairly straightforward that $\S_r$ is $\frac{4r^2}{(1-r^2)^2}$-submultiplicative (see below).  Note that for any $\ep>0$ we have that for a sufficiently small $r$, 
$\frac{4r^2}{(1-r^2)^2}<\ep.$ Let 
$$
A=\lambda\begin{pmatrix}
    1 & a^*\\
    b & ba^*
\end{pmatrix}, B=\mu\begin{pmatrix}
    1 & x^*\\
    y & yx^*
\end{pmatrix}.
$$
Then the unique nonzero eigenvalue of $A$ is $\alpha=\lambda(1+a^* b)$, the unique nonzero eigenvalue of $B$ is
$\beta=\mu(1+x^* y)$, and the unique 
nonzero eigenvalue of $AB$ is
$\gamma=\lambda\mu(1+a^*y)(1+x^* b)$.
A quick direct computation shows that
\begin{eqnarray*}
\frac{|\gamma-\alpha\beta|}{\rho(A)\rho(B)} &=&
\frac{|\gamma-\alpha\beta|}{|\alpha| |\beta|} 
= \left|\frac{\gamma}{\alpha\beta}-1
\right|\\ 
&=& \left|\frac{(1+a^*y)(1+x^* b)}{(1+a^* b)(1+x^* y)}-1\right|
\le \left|\left|\frac{(1+a^*y)(1+x^* b)}{(1+a^* b)(1+x^* y)}\right|-1\right|.
\end{eqnarray*}
Note that each of the factors $|(1+a^*b)|, |(1+a^*y)|, |(1+x^*b)|,
|(1+x^*y)|$ lies in the interval
$\left[1-r^2,1+r^2\right]$ and therefore
the value of $\left|\frac{(1+x^*b)(1+a^* y)}{(1+x^* y)(1+a^* b)}\right|$
must be between 
$\frac{(1-r^2)^2}{(1+r^2)^2}$ and
$\frac{(1+r^2)^2}{(1-r^2)^2}$.  Subtracting $1$ from these values then yields the promised estimate. 

\section{Main Results}  As mentioned in the introduction, a scaled multiplicative version of approximate submultiplicativity is more convenient in the context of unitary groups: in this case the eigenvalues always lie on the unit circle in $\mathbb{C}$ and we find it more convenient to consider the ``arc-distance" between them.  The scaling factor $\frac{1}{2\pi}$ was chosen so that the distance between two consecutive $n$-th roots of unity is $\frac{1}{n}$ (and, more generally, for $x\in[0,1)$, the distance between $1$ and $e^{2\pi x i}$ is $x$).  
\begin{definition}
     We say that the spectrum of a group $\mathcal{G}\subseteq \mathcal{M}_n(\mathbb{C})$ of unitary matrices is $\varepsilon'$-argument-submultiplicative (or $\ep'$-ASM) if for every $A,B \in \mathcal{G}$ and every $\gamma\in\sigma(AB)$, there exist elements $\alpha\in \sigma(A), \beta\in\sigma(B)$ such that
\[\frac{1}{{2\pi}}\left|\arg\left(\frac{\alpha\beta}{\gamma}\right)\right|\le\varepsilon'.\]
\end{definition}
We remark, that for unitary groups $\ep'$-ASM implies $(2\pi\ep')$-submultiplicativity and\break $\ep$-submultiplicativity implies $(\frac{1}{2}\ep)$-ASM.  This is because for complex numbers $z,w$ of modulus one we have that $|z-w|\le |\arg(z/w)|$ and $|\arg(z/w)|\le \pi |z-w|$. The second inequality can be substantially improved when $|z-w|$ is small as $\lim_{z\to w}\frac{|z-w|}{|\arg{z/w}|}=1$.  We will sometimes refer to $\frac{1}{2\pi}\left|\arg\left(\frac{z}{w}\right)\right|$ as the \emph{scaled-argument-distance} between $z$ and $w$. 

In this section we will show that $\frac{1}{2n^2}$-ASM groups of unitary $n\times n$ matrices are finite modulo their centres (hence essentially finite when irreducible).

But first we show by example, that for infinitely many $n$, there exist $\frac{2}{n^2}$-ASM groups of unitary $n\times n$ matrices that are not finite modulo their centres.  Let $p$ be an odd prime, let $D$ be any $p\times p$ unitary diagonal matrix of determinant $1$, and let $C$ to be the $p \times p$ cycle matrix, i.e.,
\[C = \begin{pmatrix}
     0 & 1 & 0&\dots & 0\\
     0 & 0 & 1 & \dots & 0 \\
     \vdots&\vdots&\ddots&\ddots&\vdots\\
     0 & 0 &\dots & 0 & 1\\
     1 & 0 & \dots & 0 & 0
 \end{pmatrix}.\]
We will use the following well-known fact about the spectra of $DC^k$.  We include a sketch of the proof for the sake of completeness.
\begin{lemma}\label{lem-spectrum}
    Suppose $k$ is some positive integer. Then for $C$ and $D$ as above we have that
    \[\sigma(DC^k) = \begin{cases}
 \sigma(D) ,& \text{ if } k\mbox{ is divisible by } p\\
 \{1, \theta, \theta^2, \dots, \theta^{p-1}\} ,&\text{ otherwise}
 \end{cases}\]
     where $\theta = e^{\frac{2\pi i}{p}}$ is a fixed primitive $p$-th root of unity.
\end{lemma}
\begin{proof}
First note that if $k$ is a multiple of $p$, then $C^k = I$, so $\sigma(DC^k) = \sigma(DI) = \sigma(D)$. Now suppose $1\le k\le p-1$. Then an easy direct computation shows that the characteristic polynomial $p(\lambda)$ of $D C^k$ is $\lambda^p-1$. 

Alternatively, it is in fact possible to see directly, that the matrices $D C^k$ and $C^k$ are similar: the matrix form of the linear map corresponding to $DC^k$ in basis $f_j=\left(\prod_{i=1}^j d_i\right) e_{j}$, $j=1,\ldots, p$  is $C^k$ (here $d_i$ denotes the $i$-th diagonal entry of $D$).  It is also 
well-known that the matrices $C^k$ and $C$ are similar by the change of basis $g_j=e_{1+(j-1)k}$,
$j=1,\ldots, p$ (here $(e_j)_{j=1}^p$ denotes the standard orthonormal basis of $\mathbb{C}^p$, we use the convention that for $j>p$ we have $e_{j}=e_{j-p}$). 
\end{proof} 

\begin{example}\label{ex-tadpole}
    Let (as above) $C$ be the $p\times p$ cycle matrix and let $\xi = e^{\frac{2\pi i}{p^2}}$ (i.e., $\xi$ is a fixed primitive $p^{2}$-th root of unity).  For a unitary diagonal matrix $D$, and integer parameters $k\in\{0, 1,2,\dots, p-1\}$, and $a_j\in \{0,1,2,\dots, p-1\}$ with $j = 1, 2, \dots, p - 1$, we define
the corresponding ``tadopole matrix" $A\in\M_{2p}(\mathbb{C})$ by
    \[A= A\left(D,k,(a_j)_{j=1}^{p-1}\right)=\begin{pmatrix}
     D C^k\\
    & 1\\
    &&\xi^{k+a_1p}\\
    &&&\xi^{2k+a_2p}\\
    &&&&\ddots\\
    &&&&&\xi^{(p-1)k+a_{p-1}p}
 \end{pmatrix}.\]
We will use notation $A_H=DC^k$ to denote the ``head" of the tadpole matrix, $D_A=D$ to denote the ``weight" of the head, and \[A_T = \begin{pmatrix}
     1\\
     &\xi^{k+a_1p}\\
     &&\ddots\\
     &&&\xi^{(p-1)k+a_{p-1}p}
 \end{pmatrix}\]
 to denote the ``tail'' of the tadpole matrix.  With this notation in mind, we write (purposefully abusing the notation somewhat):
$$
A=\begin{pmatrix} A_H & \\
& A_T \end{pmatrix}=\begin{pmatrix} D_AC^{k_A} & \\
& A_T \end{pmatrix}.
$$
\end{example}
Let $\mathcal{T}$ be the set of all such tadpole matrices (i.e., tadpole matrices corresponding to all possible choices of $D$, $k$, $(a_j)_{j=1}^{p-1}$).  We claim that $\mathcal{T}$ forms a group under matrix multiplication. Clearly, the identity matrix is in $\mathcal{T}$ and the inverses of elements from $\mathcal{T}$ are again in $\mathcal{T}$.  We are left to show that the product of any two elements of $\mathcal{T}$ is also in $\mathcal{T}$.  Let $A=A\left(D,k,(a_j)_{j=1}^{p-1}\right)$ be as above and let $B \in \mathcal{T}$ be written as
\[B= \begin{pmatrix}
     D_BC^\ell\\
     &1\\
     &&\xi^{\ell+b_1p}\\
     &&&\xi^{2\ell+b_2p}\\
     &&&&\ddots\\
     &&&&&\xi^{(p-1)\ell+b_{p-1}p}
 \end{pmatrix},\]
 with  $\ell \in \{1,2,\dots, p-1\}$ and $b_j \in \{0, 1,2,\dots, p-1\}$ for $j = 1, 2, \dots, p - 1$. Then observe that
 \[AB = \begin{pmatrix}
     D_{AB}C^r\\
     &1\\
     &&\xi^{r+c_1p}\\
     &&&\xi^{2r+c_2p}\\
     &&&&\ddots\\
     &&&&&\xi^{(p-1)r+c_{p-1}p}
 \end{pmatrix},\]
 where $D_{AB} = D_A(C^kD_BC^{-k})$, $r\equiv k+\ell-p= k+\ell  \text{ (mod p)}$, and $c_j \equiv a_j + b_j + j \text{ (mod p)}$. Note that 
 \[\xi^{jr+c_jp} = \xi^{jk+a_jp}\xi^{j\ell+b_jp} = \xi^{j(k+\ell)+(a_j+b_j)p}= \xi^{j(k+\ell-p)+(a_j+b_j+j)p}.\]
So the exponents in the tail of $AB$ are of the required form for $AB$ to belong to $\mathcal{T}$.

 \begin{theorem}\label{thm-5}
     The group $\mathcal{T}$ is $\frac{1}{2p^2}$-argument-submultiplicative. 
 \end{theorem}
 \begin{proof}
     Let $A,B\in\mathcal{T}$.  Due to symmetry we have four possible cases to consider.
 
 \textit{Case 1.} Suppose that both $A$ and $B$ are diagonal. In this case we obviously have $\sigma(AB)\subseteq\sigma(A)\sigma(B)$.
 
 \textit{Case 2.} Suppose that $A$ is diagonal, but $B$ is not. Then $k=0$ and hence $\sigma(D_AC^k) = \sigma(D_A)$. Note that \[\sigma(A) = \sigma(D_A) \cup \{1, \xi^{k+a_1p}, \xi^{2k+a_2p}, \dots, \xi^{(p-1)k+a_{p-1}p}\}\] \text{and} \[\sigma(B) = \sigma(D_BC^{\ell})\cup\{1, \xi^{\ell+b_1p}, \xi^{2\ell+b_2p}, \dots, \xi^{(p-1)\ell+b_{p-1}p}\}.\] 
The tail of $AB$ is the product of the tails of $A$ and $B$. Since these are diagonal matrices, we have that $\sigma(A_TB_T) \subseteq \sigma(A_T)\sigma(B_T)$. We are left to prove the same for the heads of $A$ and $B$. 
 Since $D_A$ and $D_B$ are unitary diagonal matrices of determinant $1$, $D_AD_B$ is also a unitary diagonal matrix of determinant $1$. By Lemma \ref{lem-spectrum} we have that
$$
\sigma(D_A D_B C^\ell) = \{1,\theta,\ldots, \theta^{p-1}\} = \sigma(D_B C^\ell),
$$
and therefore
$\sigma(D_AD_BC^\ell) \subseteq \sigma(B)$. Since $1\in \sigma(A)$, we have that $\sigma(D_AD_BC^\ell)\subseteq \sigma(A)\sigma(B)$. Hence, $\sigma(AB)\subseteq \sigma(A)\sigma(B).$  

 \textit{Case 3.} Suppose none of $A$, $B$, $AB$ are diagonal. We will again show that $\sigma(AB)\subseteq\sigma(A)\sigma(B)$. The tails must still be diagonal, so by arguments in Case 2 we have that $\sigma(A_TB_T)\subseteq \sigma(A_T)\sigma(B_T)$. It is therefore sufficient to show $\sigma(D_AC^kD_BC^{\ell})\subseteq \sigma(D_AC^k)\sigma(D_BC^{\ell})$. We find, just like in Case 2, that each of the sets $\sigma(D_A C^k), \sigma(D_B C^\ell), \sigma(D_{AB} C^r)$ is equal to $\{1,\theta,\ldots,\theta^{p-1}\}$ from when the result immediately follows.
 
 \textit{Case 4.} Suppose that neither $A$ nor $B$ is diagonal, but $AB$ is. So $k\ne 0$, $\ell=p-k$ and $r = 0$. We will first show that $\sigma(A)\sigma(B) = \{1,\xi,\xi^2,\dots,\xi^{p^2-1}\}$.
 Since neither $A$ nor $B$ is diagonal, the spectra of their heads consist of all $p$-th roots of unity.  The spectra of tails are contained in the set of all $p^2$-th roots of unity and hence $\sigma(A)\sigma(B)$ must be a subset of the set of all powers of $\xi$. We are left to show that $\sigma(A)\sigma(B)$ contains $\{1,\xi,\xi^2,\dots,\xi^{p^2-1}\}$. We know from the above observations that:
\begin{eqnarray*}
\sigma(A) &\supseteq& \sigma(A_T)= \{1, \xi^{k+a_1p}, \xi^{2k+a_2p}, \dots, \xi^{(p-1)k+a_{p-1}p}\},\\
\sigma(B) &\supseteq&  \sigma(B_H)=\{1,\theta,\ldots, \theta^{p-1}\}=\{\xi^{tp}: t=0,1,\ldots, p-1\}\\
&\supseteq&
 \{1, \xi^{(-a_j+t)p}: j=1,2,\dots,p-1,t = 0,1,\dots, p-1\}.
\end{eqnarray*}
 Then multiplying two corresponding elements of these sets will give elements in $\sigma(A)\sigma(B)$ of the form $\xi^{kj+a_jp+(-a_j+t)p} = \xi^{kj+tp}$. Since $t$ ranges over all integers in the set $[0,p-1]$ and the residue of $kj$ modulo $p$ can be any integer in $[0,p-1]$, we can obtain all powers of $\xi$; so $\sigma(A)\sigma(B) = \{1,\xi,\xi^2,\dots,\xi^{p^2-1}\}$. Since the scaled-argument-distance between consecutive powers of $\xi$ is always $\frac{1}{p^2}$, we must have that for any $\gamma \in \sigma(AB)$ there is some $\alpha\in \sigma(A)$ and $\beta\in\sigma(B)$ such that 
 \[\frac{1}{2\pi}\left|\arg\left(\frac{\alpha\beta}{\gamma}\right)\right|\le \frac{1}{2p^2}.\]
(Take $\alpha, \beta$ to be such that the product $\alpha\beta$ is the $p^2$-th root of $1$ that is closest to $\gamma$.)
 \end{proof} 

 The above theorem shows that for $n=2p$ where $p$ is an odd prime, there exists a group $\G$ of unitary $n\times n$ matrices that is $\frac{2}{n^2}$-ASM, but not finite modulo its centre.  In the next result, we show that reducing the size of $\ep'=\frac{2}{n^2}$ by a factor of $4$ does in fact force the group in question to be finite modulo its centre.
 
 \begin{theorem}\label{th-main}
     Let $\mathcal{G} \subseteq \mathcal{M}_n(\mathbb{C})$  be a group of unitary matrices.  If $\G$ is $\frac{1}{2n^2}$-argument-submultiplicative, then $\mathcal{G}$ must be finite modulo its centre. 
 \end{theorem}
 \begin{proof}
First assume with no loss of generality that $\G=\overline{\G}$, i.e., that $\G$ is compact.  Indeed, if $\G$ is $\ep'$-ASM, then, since the spectrum is continuous, so is $\overline{\G}$.  Also, if $\overline{\G}/Z(\overline{\G})$ is finite, then so is $\G/Z(\G)$.  This is because $Z(\G)\subseteq Z(\overline{\G})\cap \G$ and hence $\G/Z(\G)$ can be identified as a quotient of
$\G/(Z(\overline{\G})\cap \G)\simeq (\G Z(\overline{\G}))/Z(\overline{\G}),
$
which in turn is a subgroup of $\overline{\G}/Z(\overline{\G})$.

Suppose now, toward a contradiction, that $\G$ is not finite modulo its centre. Then for all primes $q$ we know that $\mathcal{G}$ contains some finite minimal nonabelian group, say $\mathcal{H}$, whose commutator subgroup $ [\mathcal{H} ,\mathcal{H}]$ is a $q$-group (see \cite[Lemma 2.5]{Ball}; invoking this result is the reason we needed to assume that $\G$ is compact). Let $q>n$ be a prime such that
$\frac{1}{q}<\frac{1}{2(n^2-1)}-\frac{1}{2n^2}$ (the reason for the latter requirement will become apparent later in the proof).  The structure of all finite minimal nonabelian groups has been described in detail by Miller and Morreno \cite{MM}.  They are either $p$-groups or their order is divisible by exactly two distinct primes $p$ and $q$, where $q$ is the order of its commutator subgroup (see \cite[Thm. 2.3.1.]{MR}).  In both cases, all their non-scalar irreducible representations are of size $p$  \cite[Thm. 2.3.1.]{MR}.  Since $q>n$ and clearly $p\le n$, we must therefore have that $\H$ is one of the latter (i.e., its order is divisible by two distinct primes $p$ and $q$).
 From \cite[Thm. 2.2.3]{MR} we deduce that
$\mathcal{H} = \langle X,Y \rangle$,
for matrices $X$, $Y$  of the form
\[X = \begin{pmatrix} 
    X_1\\
    &\ddots\\
    &&X_m\\
    &&&1\\
    &&&&\ddots\\
    &&&&&1
\end{pmatrix} \text{ and } Y = \begin{pmatrix}
    Y_1\\
    &\ddots\\
    &&Y_m\\
    &&&\beta_{m+1}\\
    &&&&\ddots\\
    &&&&&\beta_{\ell}
\end{pmatrix}.\]
with $m\ge 1$ and 
\[X_i =
    \begin{pmatrix}
        \theta_{i,1}\\
        &\theta_{i,2}\\
        &&\ddots\\
        &&&\theta_{i,p-1}\\
        &&&&\theta_{i,p}
    \end{pmatrix}, Y_i = \beta_i\begin{pmatrix}
     0 & 1 & 0&\dots & 0\\
     0 & 0 & 1 & \dots & 0 \\
     \vdots&\vdots&\ddots&\ddots&\vdots\\
     0 & 0 &\dots & 0 & 1\\
     1 & 0 & \dots & 0 & 0
 \end{pmatrix}\in\mathcal{M}_p(\mathbb{C}),\]
for $i=1,\ldots, m$,
where each $\theta_{i,j}$ is of order $q$, the order of each $\beta_i$ is a power of $p$, and additionally, each
$X_i$ is non-scalar and of determinant $1$.

 We now consider the spectra of $A=X^kY, B=Y^{-1}$, and $C=AB=X^k$. Since $X$ and $Y$ are block diagonal and $X$ is of determinant $1$, we have by Lemma \ref{lem-spectrum} that, 
 $$\sigma(X^kY) = \bigcup_{i=1}^m \beta_i\{1,\theta, \dots, \theta^{p-1}\} = \sigma(Y)$$ 
(where $\theta$ a fixed primitive $p$-th root of unity). Observe also that
$$\sigma(Y^{-1}) = \bigcup_{i=1}^m\beta_i^{-1}\{1,\theta, \dots, \theta^{p-1}\}.$$  
From this, since both $\sigma(X^kY)$ and $\sigma(Y^{-1})$ have cardinality at most $n$ we can see that the number of distinct elements in their product $\sigma(X^kY)\sigma(Y^{-1})$ is at most $n^2 - n - mp^2+p +mp\le n^2 - 1$ (due to repetitions). So the average scaled-argument-distance between two consecutive points in $\sigma(X^kY)\sigma(Y^{-1})$ is at least $\frac{1}{n^2-1}$. Hence there must be two consecutive points, say $\omega$ and $\eta$, whose scaled-argument-distance is at least $\frac{1}{n^2-1}$. Thus their midpoint $\gamma$ on the unit circle has scaled-argument-distance at least $\frac{1}{2(n^2-1)}$ from each $\omega$ and $\eta$ (and hence at least that much form any element of $\sigma(A)\sigma(B)$). 
As we vary $k$, note that $\sigma(X^kYY^{-1}) = \sigma(X^k)$
will contain the $q$-th root of unity
closest to $\gamma$. But then, since $\frac{1}{q}<\frac{1}{2(n^2-1)}-\frac{1}{2n^2}$, the scaled-argument-distance between an element in $\sigma(X^k)$ and any element in $\sigma(X^kY)\sigma(Y^{-1})$ will be at least $\frac{1}{2(n^2-1)}-\frac{1}{q}>\frac{1}{2n^2}$. But this then contradicts the fact that the spectrum of $\mathcal{G}$ is $\frac{1}{2n^2}$-argument-submultiplicative. Thus $\mathcal{G}$ must be finite modulo its center. 
 \end{proof}
 
\section{Further Explorations}
\subsection{Optimality of our results}
The group of tadpole matrices in Example \ref{ex-tadpole} is an example of a $\frac{2}{n^2}$-ASM groups of unitary $n\times n$ matrices that is not finite modulo its centre (here $n=2p$, where $p$ is an odd prime).  But this group is clearly not irreducible.  The group of its heads (i.e., the group of $p\times p$ matrices generated by all unitary diagonal matrices $D$ and the cycle matrix $C$) is easily seen to be an example of an irreducible group of unitary $p\times p$-matrices that is $\frac{1}{2p}$-ASM, but not essentially finite. Can we do better?  In particular: is there some constant $c$ such that for infinitely many $n$, there  exist irreducible $\frac{c}{n^2}$-ASM groups of unitary $n\times n$ matrices that are not essentially finite?

\subsection{Remarks on a representation-theoretic version of approximate submultiplicativity and linear bounds} The class ($\hat{\mathrm{s}}$) of groups $G$ with the property that all irreducible representations of all subgroups are submultiplicative was introduced in \cite{K3} and then studied further in \cite{GKOR}.  We extend this class of groups in the definition below.  Note that in the case of $\ep$-submultiplicativity we are considering all representations, whereas in the case of $\ep$-ASM we restrict ourselves to the class of unitary representations. We remark that for finite groups there is no loss of generality in this restriction, as every representation is equivalent to a unitary representation.
\begin{defn}
Let $\bfep=(\ep_n)_{n=1}^\infty$ be a sequence of nonnegative real numbers and let $G$ be an abstract group.  We say that $G$ is in class $(\epws)$ if for every $n$, every subgroup $H$ of $G$, and every irreducible representation 
$\rho\colon H\to\mathcal{M}_n(\mathbb{C})$, the image $\rho(H)$ is $\ep_n$-submultiplicative.  We say that $G$ is in class $(\epwas)$ if for every $n$ and every subgroup $H$ of $G$, the image of every irreducible unitary representation $\rho\colon H\to\mathcal{M}_n(\mathbb{C})$ is $\ep_n$-ASM.

We also define the corresponding notions for any fixed representation: a representation $\rho\colon G\to\M_n(\mathbb{C})$ is $\epws$, or $\epwas$ respectively (for the latter we additionally assume that $\rho$ is unitary), if for every subgroup $H$ and every irreducible subrepresentation of $H$, $\rho_W\colon H\to \operatorname{End}(W)\simeq \M_m(\mathbb{C})$, the image $\rho_W(H)$ is $\ep_m$-submultiplicative, or $\ep_m$-ASM, respectively (here $W$ is a minimal invariant subspace for $\rho(H)$ of dimension $m$).  
\end{defn}

Assume now that $\bfep=(\ep_n)_{n=1}^\infty$ is a sequence of positive numbers such that for every prime $p$, $\ep_p<\frac{1}{2p}$.  For a prime $p$ also define $\delta_p=\frac{1}{2p}-\ep_p$ and let $Q(p)=Q_{\bfep}(p)$ denote the set of all primes $q$ for which 
$$
\bigcup_{j=0}^{p-1} \left(\frac{2j+1}{2p}-\delta_p, \frac{2j+1}{2p}+\delta_p\right) \cap \left\{\frac{k}{q}: k=1,\ldots, q-1\right\}=\emptyset.
$$
In other words, $q\in Q(p)$ precisely when the distance from any fraction $\frac{k}{q}$, $k=1\ldots, q-1$ to the closest fraction of the form $\frac{j}{p}$, $j=0,\ldots, p$ is smaller than $\ep_p$.  Note that if $q>\frac{1}{2\delta_p}$, then $q\not\in Q(p)$.  Therefore, the set $Q(p)$ is always finite.

Assume now that $G$ is a finite minimal nonabelian group whose order is divisible by two primes $p$, $q$ (not necessarily distinct) and that $[G,G]$ is a $q$-group (in the language of \cite{MR} we would say that $G$ is the $(p,q,f)$-group for some irreducible divisor $f$ of $x^p-1\in\mathbb{Z}_q[x]$).
It is then easy to see (using the same ideas as in the proof of Theorem \ref{th-main}) that $G$ is in the class $(\epwas)$ if and only if $q\in Q(p)$. 
Since $Q(p)$ is finite, we can therefore conclude that any compact matrix group (viewed as a representation of itself in the obvious way) satisfying $\epwas$ must be finite modulo its centre. 

We also remark that there are easy examples of sequences $\bfep$ satisfying the above properties such that every $Q(p)$ also contains primes $q$ different from $p$.  In these cases, the corresponding $(p,q,f)$-groups will not be nilpotent.  Hence, for such $\bfep$, the class $(\epwas)$ strictly extends the class $(\widehat{\mathrm{s}})$. Hence, it is natural to ask the following question.
\begin{quest} Are there sequences $\bfep$ for which the class $(\epwas)$ is contained in (or perhaps even coincides with) a well-known class of finite groups (e.g, $M$-groups, supersolvable groups, etc.)?
\end{quest}

\section*{Acknowledgments}
We would like to express our gratitude to the referee.  Their suggestions significantly improved the paper.

M.~Mastnak was supported in part by the NSERC Discovery Grant 371994-2019.
Most of the research presented in the paper was done as part of an undergraduate summer research project.  The authors thank Saint Mary's University and NSERC for their generous support.

\bibliographystyle{plain}

\end{document}